\documentclass[a4paper,11pt]{article}
\usepackage[T1]{fontenc}
\usepackage{rotating}

\usepackage[title]{appendix}
\usepackage{url}
\usepackage{upgreek}
\usepackage{amsmath} 
\usepackage{amsthm}
\usepackage{amsfonts}
\usepackage{mathtools}
\usepackage{amssymb}
\usepackage{enumitem}
\usepackage{gensymb}

\usepackage{color}

\usepackage{algorithm}
\usepackage{algpseudocode}

\def\bv{\bar{v}}

\newcommand{\ZZ}{\mathbb{Z}}

\DeclareMathOperator{\Aut}{Aut}

\DeclareMathOperator{\V}{V}
\DeclareMathOperator{\E}{E}
\DeclareMathOperator{\LCF}{LCF}

\newcommand{\cR}{{\mathcal{R}}}
\newcommand{\cX}{{\mathcal{X}}}

\usepackage{graphicx}

\usepackage{tikz}
\usetikzlibrary{calc}
\usetikzlibrary{backgrounds}
\usetikzlibrary{intersections}
\usetikzlibrary{decorations.markings}
\usetikzlibrary{arrows.meta}
\usetikzlibrary{decorations.pathreplacing}

\usepackage{authblk}
\usepackage{subcaption}
 
\usepackage[text={170mm,254mm}]{geometry}

\theoremstyle{plain}
\newtheorem{theorem}{Theorem}
\newtheorem{lemma}[theorem]{Lemma}

\theoremstyle{definition}
\newtheorem{definition}{Definition}
\newtheorem{problem}[theorem]{Problem}

\usepackage{ifxetex,ifluatex}
\if\ifxetex T\else\ifluatex T\else F\fi\fi T%
  \usepackage{fontspec}
\else
  \usepackage[T1]{fontenc}
  \usepackage[utf8]{inputenc}
  \usepackage{lmodern}
\fi

\usepackage{hyperref}

\title{Computing the Hamiltonian compression factors of cubic graphs}

\author[1,2]{Marston Conder\thanks{Supported by New Zealand's Marsden Fund, grant number UOA2320}}
\author[3]{Gregor Poto\v{c}nik}
\author[4,5]{Primo\v{z} Poto\v{c}nik\thanks{Supported by Slovenian Research and Innovation Agency, project numbers P1-0294 and J1-4351.}}

\affil[1]{University of Auckland, Department of Mathematics, Auckland, New Zealand}
\affil[2]{University of Primorska, Koper, Slovenia}
\affil[3]{Gimnazija Vi\v{c}, Ljubljana, Slovenia}
\affil[4]{University of Ljubljana, Faculty of Mathematics and Physics, Ljubljana, Slovenia}
\affil[5]{Institute of Mathematics, Physics and Mechanics, Ljubljana, Slovenia}

\begin{document}

\maketitle

\begin{abstract}
We present an algorithm for computing Hamiltonian cycles that are invariant under a graph automorphism acting on them as a rotation.
We also present an application of this algorithm for computing the Hamiltonian compression factor of a graph, that is, the largest order of an automorphism preserving some
Hamiltonian cycle and acting on it as a rotation. As an example, we compute the Hamiltonian compression factors of all cubic edge-transitive graphs on up to $10{,}000$ vertices,
with the exception of two graphs, which are not Hamiltonian, and $98$ graphs (the smallest having $2304$ vertices) for which only a lower bound for the compression factor is given. As a byproduct, we obtain shortest LCF codes for each of these graphs (except for the two non-Hamiltonian ones; for the $98$ unresolved graphs, the codes obtained are the shortest among those we found).
\end{abstract}

\section{Introduction}

Finding a Hamiltonian cycle in a given graph (that is, a cycle which visits every vertex of the graph exactly once) is a well-known computationally difficult problem.
Its decision analogue (determining whether a graph has a Hamiltonian cycle) is NP-complete, and remains such even when restricted to several important subclasses,
such as graphs of maximum degree $3$, bipartite graphs, planar graphs etc. Several exponential-time algorithms that enumerate all Hamiltonian cycles of a graph have
been presented in the literature, both for general graphs (see, for example, \cite{Rub}) as well as for {\em subcubic graphs} (that is, those of maximum degree 
at most $3$)  \cite{Epp}. An implementation of an algorithm ({\tt cubeham.c}) that enumerates all Hamiltonian cycles for subcubic graphs,
as well as an implementation ({\tt hamheuristic.c}) of a heuristic algorithm that tries to find a Hamiltonian cycle in a general graph can be found in the {\tt nauty} package \cite{nauty}.

When considering graphs possessing non-trivial automorphisms, especially those with many automorphisms, rather than just ask for existence of a Hamiltonian cycle,
one might perhaps want to find a Hamiltonian cycle which is invariant under a non-trivial automorphism of a graph.
This line of inquiry was suggested already by Coxeter and Frucht in the 1970s (see \cite{Frucht} for explanation), and has resurfaced recently in \cite{HamComp,KMR},
where the notions of a {\em $k$-symmetric Hamiltonian cycle} and the {\em Hamiltonian compression factor of a graph} were coined.  

\begin{definition}
\label{def:main}
Let $\Gamma$ be a finite simple graph on $n$ vertices and let $C=v_0v_1\ldots v_{n-1}$ be a Hamiltonian cycle of $\Gamma$. 
If there exists an automorphism $g$ of $\Gamma$ such that ${v_i}^g = v_{i+\frac{n}{k}}$ for some positive divisor $k$ of $n$ and for all indices $i$ (computed modulo $n$),
then $C$ is said to be {\em $k$-symmetric} and $g$ is called a {\em $k$-fold rotation} of $C$.
The largest integer $k$ for which $C$ is $k$-symmetric is called the {\em compression factor of $C$} and denoted by $\kappa(\Gamma,C)$,
and the largest integer $k$ for which there exists a $k$-symmetric Hamiltonian cycle of  $\Gamma$ is called the {\em (Hamiltonian) compression factor of $\Gamma$} and denoted by $\kappa(\Gamma)$.
If $\Gamma$ has no Hamiltonian cycles, then $\kappa(\Gamma)=0$. 
\end{definition}

For convenience, let us call a $\kappa(\Gamma)$-fold rotation of a Hamiltonian cycle $C$ a {\em compression-optimal automorphism}, and let us call $C$ 
a {\em compression-optimal Hamiltonian cycle}.

A permutation $g$ on a set $\Omega$ is said to be {\em semiregular} if
no non-identity element of $\langle g \rangle$ fixes a point in $\Omega$. Note that by this 
definition the trivial permutation is semiregular, which might differ from the definition of semiregularity
elsewhere, where the trivial permutation is explicitly excluded. 
Note also that a rotation of a Hamiltonian cycle $C$ of a graph $\Gamma$
preserves the edge-set $\E(C)$ and acts on $\V(\Gamma)$ as well as on $\E(C)$ as a semiregular permutation.
Conversely, every automorphism of $\Gamma$ that preserves a Hamiltonian cycle $C$ and acts as a semiregular permutation
on $\E(C)$ and on $\V(\Gamma)$ is a rotation of $C$. (We shall
call the automorphisms of $\Gamma$ that act semiregularly on $V(\Gamma)$ {\em semiregular automorphisms} of $\Gamma$.)
Hence we obtain the following lemma:

\begin{lemma}
\label{lem:init}
Let $C$ be a Hamiltonian cycle of a graph $\Gamma$ of order $n$. Then the compression factor of $C$ is the maximum order of
an automorphism of $\Gamma$ that preserves $\E(C)$ and acts semiregularly both on $\V(\Gamma)$ and on $\E(C)$.
\end{lemma}

As was pointed out in \cite{HamComp}, in the case of cubic (that is, $3$-regular) graphs there is a tight connection between
the compression factor and the so-called LCF notation, introduced by Lederberg and later studied by Coxeter and Frucht \cite{Frucht},
 which assigns to a cubic graph $\Gamma$ with a Hamiltonian cycle $C:=v_0v_1\ldots v_{n-1}$ the sequence 
 $$\LCF(\Gamma,C):= [s_0,s_1,\ldots,s_{n-1}]$$
 of integers $s_i$,  for $- n/2 < s_i \le n/2$, such that  $v_i \sim v_{i+{s_i}}$ 
 for every $i\in \ZZ_n$. (Here and in what follows the indices are computed within the group $\ZZ_n$, that is, modulo $n$, and the symbol `$\sim$'
 represents the adjacency relation in the graph under consideration.) 
Note that the LCF notation determines a cubic graph completely, as the neighbours of a vertex $v_i$ are $v_{i-1}$, $v_{i+1}$ and $v_{i+{s_i}}$.

If the LCF notation $[s_0,s_1,\ldots,s_{n-1}]$ is $k$-periodic for some positive divisor $k$ of $n$ in the sense that $s_{i+n/k} = s_i$ for every $i\in \ZZ_n$,
then the sequence is usually written as $[s_0,s_1,\ldots,s_{n/k-1}]^k$, yielding a `compressed' LCF notation with the `compression factor' being $k$.
The following easy observation therefore justifies the use of the term `Hamiltonian compression factor' of a cubic graph:

\begin{lemma}
\label{lem:LCF}
Let $C$ be a Hamiltonian cycle of a cubic graph $\Gamma$. Then the compression factor $\kappa(\Gamma,C)$ equals the largest integer $k$
such that $\LCF(\Gamma,C)$ is $k$-periodic. In particular, the Hamiltonian compression factor $\kappa(\Gamma)$ equals the largest
integer $k$ such that $\Gamma$ admits a $k$-periodic LCF notation.
\end{lemma}

The main purpose of this paper is to address the following problem:

\begin{problem}
\label{prob:main}
Find a practical algorithm which takes a (finite, connected) cubic graph $\Gamma$ as an input and returns the Hamiltonian compression factor $\kappa(\Gamma)$
together with a $\kappa(\Gamma)$-periodic LCF notation for $\Gamma$.
\end{problem}

Note that this problem is at least as hard as the Hamiltonian cycle existence decision problem, and indeed, in the case of graphs that are not Hamiltonian or
that only possess Hamiltonian cycles with no non-trivial rotations  (that is, for graphs with Hamiltonian compression factor  $1$),
our solution of this problem reduces to the use of existing algorithms for finding Hamiltonian cycles in cubic graphs.
For graphs with large Hamiltonian compression, due to the extensive use of the symmetry properties of the graph,
our algorithm not only finds an `optimal' Hamiltonian cycle, but also finds it faster than general algorithms (not using symmetry) would.
In this sense, the algorithm proposed in this paper can be viewed as an improvement of existing algorithms for finding Hamiltonian cycles.

Our approach is developed in Section~\ref{sec:algo}, and the resulting algorithm is presented and discussed in Section~\ref{sec:pseudo}. As an application, we compute the Hamiltonian compression factor of (almost all)
cubic edge-transitive graphs on at most $10{,}000$ vertices in Section~\ref{sec:CET}. The census of these graphs was first compiled by the first-named author of this paper in 2011 for the arc-transitive case, and in 2018 (jointly with the last-named author)  for the semisymmetric case. 
The censuses first appeared at \cite{ConderWebCAT,ConderWebCSS}, but are  now available for download also from \cite{FosterGraphSym} and \cite{ZenodoLCF}.

\section{The cycle lifting lemma}
\label{sec:algo}

A naive approach to determining a Hamiltonian compression factor of a graph would first try to
find all Hamiltonian cycles of a graph, for each of them compute all rotations,  and then choose one
with a rotation of largest order. The main drawback of this approach is of course the algorithmic complexity
 of finding all Hamiltonian cycles of a given graph.

Instead of this naive approach, we shall adopt the strategy mentioned in \cite[Section 1.3]{HamComp} and
 use the fact that existing algorithms of computational group theory
(and their implementations in several computer algebra systems)
allow us to determine the automorphism group of a given graph rather quickly.
Since rotations of Hamiltonian cycles all act semiregularly on the vertices (see Lemma~\ref{lem:init}),
we may disregard automorphisms that are not semiregular.
Also since a conjugate $g^h = h^{-1} g h$ of a rotation $g$ of a Hamiltonian cycle $C$
is a rotation of the Hamiltonian cycle $C^h$,
it suffices to consider one automorphism from each conjugacy class of semiregular automorphisms.
Moreover, as clearly follows from Lemma~\ref{lem:main} below, given a semiregular automorphism $g$ of order $k$, our algorithm
can check (without any additional cost) not only whether $g$  but also whether any other generator of the cyclic group $\langle g \rangle$ 
acts as a $k$-fold rotation of some Hamiltonian cycle. 
The first step of our algorithm is as follows:
\medskip

\noindent
{\bf Step 1:} For a graph $\Gamma$,
 find a complete set of representatives of $\Aut(\Gamma)$-conjugacy classes of semiregular cyclic subgroups of $\Aut(\Gamma)$ (including the trivial subgroups)
 and let $\cR$ be a set containing one generator from each of these representative subgroups.
The set $\cR$ will then be called a {\em reduced set of candidate rotations}, for short.
 (Note that by definition of semiregular automorphisms, any such set of candidate rotations always contains the trivial automorphism.)
\medskip

There are two possible computational bottlenecks in Step~1. It is well known that the determination
of the automorphism group of a graph (say, by computing its generating set) is at worst quasi-polynomial 
in the number of vertices for general graphs, and that it can be determined in polynomial-time for the case of graphs 
of bounded valence. For practical purposes, of course, these theoretical time complexities are not as important 
as how implementations of different algorithms behave in the relevant range. Fortunately, methods that are currently available
in several computer algebra systems (such as {\sc{Gap} \cite{Gap}} or {\sc{Magma}~\cite{magma}}) and code libraries (such as {\sc{Nauty}~\cite{nauty}})
perform extremely well for graphs with up to several thousand vertices.

A more problematic part of Step 1 is the computation of the conjugacy classes of cyclic subgroups, and the reason lies
in the fact that whatever the time complexity of the conjugacy classes problem is (and it is well known to be polynomial),
it will be a function of the order of the automorphism group rather than the order of the graph. It is known
that the order of the automorphism group of a graph (even when restricted to a nice class,  such as connected vertex-transitive cubic graphs)
can grow exponentially with the order of the graph. This part of Step~1 might therefore prove to be problematic for graphs
with very large automorphism groups. Fortunately, for many classes of graphs, and in particular, by a celebrated theorem of Tutte~\cite{tutte}
on cubic symmetric graphs, and the analogous theorem of Goldschmidt~\cite{Gold80} for cubic semisymmetric  graphs, the order of the automorphism group of a
connected cubic edge-transitive graph is bounded by a linear function of the order of the graph; see~\cite{TAMS,PSV} for a detailed discussion
on the problem of bounding the order of the automorphism group in highly symmetric graphs.

In the second step, it remains to iterate through the reduced set of candidate rotations:
\medskip

\noindent
{\bf Step 2:} For a reduced set $\cR$ of candidate rotations of $\Gamma$,
iterate through the elements $g$ of $\cR$ (in decreasing order of $|g|$),
and at each step check whether there exists a Hamiltonian cycle  upon which a generator of $\langle g\rangle$  acts as a rotation.
If such a Hamiltonian cycle exists, exit and return the order of $g$ as the Hamiltonian compression factor of the graph.
If the iteration terminates without finding any appropriate Hamiltonian cycles, then $\Gamma$ is not Hamiltonian (since
in the last step of the iteration $g$ is the trivial automorphism, which clearly acts as a rotation of every
Hamiltonian cycle of $\Gamma)$.
\medskip

Clearly, the core of the problem of Step~2 is deciding whether for a given semiregular automorphism $g$ of $\Gamma$ there exists
a Hamiltonian cycle upon which $g$ acts as a rotation. If $g$ happens to be trivial, then this is equivalent to deciding whether $\Gamma$
is Hamiltonian, which is known to be an {\rm NP}-complete decision problem. But, as we will now show, if $g$ has a relatively small
number of orbits (and hence has a relatively large order), then the problem becomes considerably easier. In particular, for graphs
with relatively large Hamiltonian compression factor $\kappa(\Gamma)$, Step~2 can be carried out efficiently. On the other hand, 
if $\kappa(\Gamma)$ is small, then our algorithm will yield a lower bound on $\kappa(\Gamma)$ but might fail to determine the exact value.

Let us start a detailed explanation of our approach to Step~2 by introducing some notation pertaining to the quotients of graphs and the `cycle lifting technique', whose importance for finding Hamiltonian cycles is well known (see, for example, \cite[Section 1]{KMR}).

 Let $g$ be a semiregular automorphism of a graph $\Gamma$ having $m$ orbits of size $k$,
 and let $T=\{\alpha_0^0,\alpha_0^1, \ldots, \alpha_0^{m-1}\}$ be an ordered set, containing precisely one representative 
 of each orbit of $\langle g \rangle$. We shall refer to such a set $T$ as a {\em transversal} for $g$. 
 By letting $\alpha_r^i = (\alpha_0^i)^{g^r}$ for every $r \in \ZZ$ and $i\in \{0,1,\ldots,m-1\}$,
we can write $g$ as a product of disjoint cycles as follows:
 \begin{equation}
\label{eq:g}
g \> = \> (\alpha_0^0\, \alpha_1^0\, \ldots \, \alpha_{k-1}^0)\,(\alpha_0^1\, \alpha_1^1\, \ldots \, \alpha_{k-1}^1) \, \ldots \, (\alpha_0^{m-1}\, \alpha_1^{m-1}\, \ldots \, \alpha_{k-1}^{m-1}).
\end{equation}
For convenience, we shall consider each subscript $t\in \{0,\ldots, k-1\}$ as an element of the ring $\ZZ_k$,
and each superscript $i\in \{0,1,\ldots, m-1\}$ as an element of $\ZZ_m$, and then
 always compute with subscripts and superscripts modulo $k$ and $m$, respectively.

The semiregular automorphism $g$, together with the explicit choice of the transversal $T$, determines 
two mappings $\pi \colon V(\Gamma) \to \ZZ_m$ and $\sigma \colon V(\Gamma) \to \ZZ_k$ defined by
\begin{equation}
\label{eq:pisi}
 \pi(\alpha_t^i) \> =\>  i,\qquad  \sigma(\alpha_t^i) \> =\>  t.
\end{equation}

We can also define the {\em quotient graph} $\Gamma/\langle g \rangle$  as the graph with vertex-set $\ZZ_m$
and two distinct vertices  $i,j \in \ZZ_m$ adjacent in $\Gamma/\langle g \rangle$ if and only if $\alpha_t^i$ is adjacent to $\alpha_r^j$ for some $t,r\in \ZZ_k$.

More generally, for every pair $i,j \in \ZZ_m$, let
\begin{equation}
\label{eq:zeta}
 \zeta(i,j)\> = \> \{s \in \ZZ_k \mid \alpha_0^i \sim \alpha_s^j \},
\end{equation}
and observe that whenever $i\not =j$, the set $\zeta(i,j)$ is non-empty if and only if $i$ is adjacent to $j$ in $\Gamma/\langle g \rangle$.
Observe also that
since $g$ preserves adjacency in $\Gamma$ and since $\alpha_t^i = (\alpha_0^i)^{g^t}$ by definition,
we see that $\alpha_t^i \sim \alpha_r^j$ if and only if $\alpha_0^i \sim \alpha_{r-t}^j$ if and only if $r-t \in \zeta(i,j)$.
This can be expressed in terms of the functions $\pi$ and $\sigma$ as follows:
\begin{equation}
\label{eq:recon}
 u \sim v \quad \Longleftrightarrow \quad \sigma(v)-\sigma(u) \in \zeta(\pi(u),\pi(v)), \qquad\hbox{ for every } \> u,v\in \V(\Gamma).
\end{equation}
This shows that the adjacency relation in $\Gamma$ can be reconstructed from the functions $\pi$, $\sigma$ and $\zeta$.
Observe also that the functions $\pi$, $\sigma$, and $\zeta$ depend solely on the choice of the semiregular element $g$
and the transversal $T$, and of course, the dependence on $T$ is only up to suitable modifications of these functions
arising from the permutation of the superscripts $i \in \ZZ_m$ and cyclic shifts of the subscripts $t$ in $\alpha_t^i$ for each fixed $i$.
We are now ready to state and prove a lemma that will serve as a basis for our computational approach.

\begin{lemma}
\label{lem:main}
Let $\Gamma$ be a connected graph of order $n$, let $g\in \Aut(\Gamma)$ be a
semiregular automorphism of order $k$ with $m$ orbits, written as in (\ref{eq:g}), let $c$ be an integer coprime to $k$,
 and let  the functions $\pi, \sigma$ and $\zeta$ be as in (\ref{eq:pisi}) and (\ref{eq:zeta}).
Suppose that $\Gamma$ possesses a Hamiltonian cycle $C=v_0v_1\ldots v_{n-1}$ such that
$g^c$ is a $k$-fold rotation of $C$, and let $\bv_i = \pi(v_i)$ for every $i\in \ZZ_n$. Then the conditions {\rm (A)} and {\rm (B)} below hold:
\begin{enumerate}
\item[{\rm (A)}]
$m\ge 3$ and $\bv_0\bv_1 \ldots \bv_{m-1}$ is a Hamiltonian cycle in $\Gamma/\langle g \rangle$, or
$m= 2$ and $\bv_0\bv_1$ is an edge of  $\Gamma/\langle g \rangle$, or
$m=1$; 
\item[{\rm (B)}] there exist elements $z_i \in \zeta(\bv_i,\bv_{i+1}),\> i\in \ZZ_m,$ such that
\begin{equation}
\label{eq:sum}
 \sum_{i=0}^{m-1} z_i  \> = \> c.
\end{equation}
\end{enumerate}

Conversely, suppose that $\bv_0, \bv_1, \ldots, \bv_{m-1}$ are vertices of $\Gamma/\langle g \rangle$ such that
the conditions {\rm (A)} and {\rm (B)} above are fulfilled. Then there exists a Hamiltonian cycle $v_0v_1\ldots v_{n-1}$ in $\Gamma$
for which $g^c$ is a $k$-fold rotation and such that $\pi(v_i) = \bv_i$ for every $i\in \{0,\ldots, m-1\}$.
\end{lemma}

\begin{proof}
Suppose first that $\Gamma$ possesses a Hamiltonian cycle $C=v_0v_1\ldots v_{n-1}$ for which $g^c$ is a $k$-fold rotation.
Let $\alpha_t^j$, $t\in \ZZ_k$, $j\in \ZZ_m$, be an arbitrary vertex of $\Gamma$, and let $i\in \ZZ_n$ be such that $v_i = \alpha_t^j$.
 Since $g^c$  is a $k$-fold rotation of $C$ and $mk=n$, it follows that 
 $v_i^{g^c} = v_{i+m}$. On the other hand, $v_i^{g^c} = (\alpha_t^j)^{g^c} = \alpha_{t+c}^j$, and thus $\pi(v_{i+m}) =  \pi(v_i^{g^c}) = \pi(\alpha_{t+c}^j) = j = \pi(v_i)$
and $\sigma(v_{i+m}) =  \sigma(v_i^{g^c}) = \sigma(\alpha_{t+c}^j ) =  t+c = \sigma(v_i)+c$. In short,
 \begin{equation}
\label{eq:sigmam}
 \pi(v_{i+m}) = \pi(v_{i})  \quad \hbox{ and } \quad \sigma(v_{i+m}) = \sigma(v_i)+c
\end{equation}
 for every $i\in \ZZ_n$. Since $C$ is a Hamiltonian cycle, it intersects every orbit of $\langle g \rangle$, implying that
 the set $\pi(\V(\Gamma)) = \{\pi(v_i) : i \in \ZZ_n\}$ covers all the elements of $\ZZ_m$. On the other hand,
 by what we just showed, we see that $\ZZ_m = \{\pi(v_i) : i \in \ZZ_n\} = \{\pi(v_i) : i \in \{0,\ldots,m-1\}\} $.
 In particular, the elements $\pi(v_0), \pi(v_1), \ldots, \pi(v_{m-1})$ 
 are pairwise distinct, and  since $v_i$ is adjacent to $v_{i+1}$ in $\Gamma$ for all $i\in \ZZ_n$, this implies that so is
 $\pi(v_i)$ to $\pi(v_{i+1})$
 in $\Gamma/\langle g \rangle$ for every $i\in \ZZ_m$.
 In particular, if $m\ge 2$, then $\pi(v_0)\pi(v_1)$ is an edge of $\Gamma/\langle g \rangle$, and if $m\ge 3$, then $\pi(v_0)\pi(v_1)\ldots \pi(v_{m-1})$
  is a Hamiltonian cycle in $\Gamma/\langle g \rangle$. This proves condition (A).
 
To prove condition (B), let
 $z_i  =   \sigma(v_{i+1}) -\sigma(v_i),  i\in \ZZ_n$,
 and observe that $z_i \in \zeta(\pi(v_i),\pi(v_{i+1}))$, by (\ref{eq:recon}).
By (\ref{eq:sigmam}), we see that  $\sigma(v_m) = \sigma(v_0) + c$, and thus
$$
\sum_{i=0}^{m-1} z_i \> = \> \sigma(v_{m}) - \sigma(v_0) \> = \>  c,
$$
as required.

Let us now
suppose that $\bv_0, \bv_1, \ldots, \bv_{m-1}$ are vertices of $\Gamma/\langle g \rangle$ (recalling that $\V(\Gamma/\langle g \rangle) = \ZZ_m$) such that
the conditions {\rm (A)} and {\rm (B)} above are fulfilled.
For convenience, we shall extend the definition of $\bv_j$ to all integers $j$ by letting $\bv_{rm+i} = \bv_i$ for all integers $r$ and $i\in\{0,1,\ldots,m-1\}$.

By~(B), there exist $z_i \in \zeta(\bv_i,\bv_{i+1})$ for every $i\in \{0,\ldots, m-1\}$ such that $\sum_{i=0}^{m-1} z_i = c$, and $c$ is coprime to $k$.
If we extend the definition of $z_i$ by letting 
$
z_{rm+i} = z_i
$ 
for all integers $r$ and $i\in\{0,\ldots,m-1\}$, then  the convention that $\bv_{rm+i} = \bv_i$ implies that $z_i \in \zeta(\bv_i,\bv_{i+1})$ for all integers $i$.

We can now define the vertices $v_i$, $i\in \ZZ$, by letting
$$
 v_i\> = \> \alpha^{\bv_i}_{s_i}\> \hbox { where } s_0 = 0 \hbox { and }  s_i = z_0 + z_1 + \ldots + z_{i-1}.
$$
We claim that $v_0v_1\ldots v_{n-1}$ is then a Hamiltonian cycle of $\Gamma$ for which $g^c$ is a $k$-fold rotation.
To see that, note first that (\ref{eq:recon}), together with the definition of the vertices $v_i$ and the fact that $s_{i+1} - s_i = z_i$, implies that
$$
v_i \sim v_{i+1} \Leftrightarrow z_i \in \zeta(\bv_i,\bv_{i+1}),
$$
which is fulfilled for all integers $i$, as already established. Moreover, 
$$
 v_{n} =  \alpha^{\bv_n}_{s_n} = \alpha^{\bv_{km}}_{s_{km}} =   \alpha^{\bv_{0}}_{s_0 + kc} =  \alpha^{\bv_{0}}_{s_0}  = v_0,
$$
where we have used the fact that the values of $z_i$ (and thus of $s_i$) are computed in $\ZZ_k$.
It remains to verify that the vertices $v_0, v_1, \ldots, v_{n-1}$ are pairwise distinct. Write any index $i\in \{0,\ldots,n-1\}$ uniquely as $i = rm + i'$ with $r\in \{0,\ldots,k-1\}$ and $i'\in \{0,\ldots,m-1\}$. By the periodicity $\bv_{rm+i'} = \bv_{i'}$ and the relation $s_{rm+i'} = s_{i'} + rc$ (which holds since the sequence $(z_i)_i$ is $m$-periodic and $z_0+z_1+\cdots+z_{m-1}=c$), we have $v_i = \alpha^{\bv_{i'}}_{s_{i'}+rc}$. Suppose $v_i = v_j$ with $j = r'm + j'$. Then $\bv_{i'} = \bv_{j'}$, and since by~(A) the elements $\bv_0, \bv_1, \ldots, \bv_{m-1}$ are pairwise distinct (as vertices of a Hamiltonian cycle in $\Gamma/\langle g\rangle$, or an edge, or a single vertex, depending on $m$), we get $i' = j'$. Then $s_{i'} + rc \equiv s_{i'} + r'c \pmod{k}$, that is, $(r-r')c \equiv 0 \pmod{k}$, and since $\gcd(c,k) = 1$, this forces $r = r'$ and hence $i = j$. This shows that
$v_0v_1\ldots v_{n-1}$ is indeed a Hamiltonian cycle.

Finally, recall that the sequence  $(z_i)_i$ is periodic with period $m$, implying that
$$
 s_{i+m} \> = \> s_i + (z_i + \ldots + z_{i+m-1}) \> = \> s_i + (z_0+\ldots+z_{m-1}) \> = \> s_i + c.
$$
Thus
$$
 v_{i+\frac{n}{k}} = v_{i+m} = \alpha^{\bv_{i+m}}_{s_{i+m}} = \alpha^{\bv_i}_{s_i + c} = (\alpha^{\bv_i}_{s_i})^{g^c} = (v_i)^{g^c},
 $$
 which shows that $g^c$ is a $k$-fold rotation of the Hamiltonian cycle $v_0v_1\ldots v_{n-1}$.
\end{proof}

\section{The algorithm}
\label{sec:pseudo}

The pseudocode in Algorithm~\ref{alg:hcf}, the correctness of which is supported by Lemma~\ref{lem:main}, summarises our approach to the computation of the Hamiltonian compression factor.
The algorithm takes a connected graph $\Gamma$ as an input and returns a triple $(\kappa,C,\rho)$ where $\kappa=\kappa(\Gamma)$ is the Hamiltonian compression factor of $\Gamma$, $C$
is a compression-optimal Hamiltonian cycle and $\rho\in \Aut(\Gamma)$ is a $\kappa$-fold rotation of $C$; in case $\Gamma$ is not Hamiltonian, the algorithm returns $(0,\bot,\bot)$.

\begin{algorithm}[t]
\caption{Computing the Hamiltonian compression factor}
\label{alg:hcf}
\begin{algorithmic}[1]
\Require A finite connected graph $\Gamma$ on $n$ vertices.
\Ensure $\kappa(\Gamma)$, $\kappa(\Gamma)$-symmetric Hamiltonian cycle $C$ (if one exists), and $\kappa(\Gamma)$-fold rotation of $C$.

\State Compute $A=\Aut(\Gamma)$.
\State Let $\mathcal R$ contain one generator from each $A$-conjugacy class of semiregular
cyclic subgroups of $A$.
\State Order $\mathcal R$ by decreasing order of its elements.
\For{$g\in\mathcal R$}
    \State Let $k$ be the order of $g$ and let $m=\frac{n}{k}$.
    \State Choose orbit representatives
    $\alpha_0^0,\alpha_0^1,\ldots,\alpha_0^{m-1}$ for $\langle g\rangle$ and let 
        $\alpha_t^i=(\alpha_0^i)^{g^t}$, for $t\in\ZZ_k, i\in\ZZ_m$.
    \State Compute the quotient graph $Q=\Gamma/\langle g\rangle$ and the sets
    \[
        \zeta(i,j)=\{s\in\ZZ_k:\alpha_0^i\sim \alpha_s^j\}\> \hbox{ for }\> i,j \in \ZZ_m.
    \]
   \If{$m=1$}
     \State Set  $\cX:= \{\bv_0\}$ (the singleton containing the only vertex of $Q$).
   \EndIf
   \If{$m=2$}
        \State Set  $\cX:= \{\bv_0\bv_1\}$ (the singleton containing the only edge of $Q$).
   \EndIf
   \If{$m\ge 3$}
    \State Let $\cX$ be the set of all Hamiltonian cycles of $Q$.
   \EndIf
    \For{each $\bar v_0\bar v_1\ldots\bar v_{m-1}$ in $\cX$}
       \If{exist $z_i \in \zeta(\bv_i,\bv_{i+1})$, such that $c:=\sum_{i=0}^{m-1}z_i$ is coprime to $k$}
          \State Lift  $\bar v_0\bar v_1\ldots\bar v_{m-1}$ using $z_0,\ldots,z_{m-1}$ to obtain a Hamiltonian cycle $C$ of $\Gamma$.    
          \State Set $\rho:=g^c$. 
          \State
          \Return $\bigl(k,C,\rho \bigr)$.
       \EndIf
    \EndFor

\EndFor

\State \Return $(0,\bot,\bot)$.
\end{algorithmic}
\end{algorithm}

Let us now discuss further details and variations of the pseudocode that make the computations faster.
As already discussed above, Lines 1 and 2 rely heavily on the efficiency of the algorithm for computing automorphism group of a graph, and then, for finding
conjugacy classes of the automorphism group. For this task, we have relied on {\sc Magma}, but could have just as easily used {\sc GAP}, or any other
similar computer algebra software. Even though these computations are difficult in general, they did not represent a bottleneck for our test case of cubic edge-transitive graphs.
We concede, however, that this could become an issue for graphs with very large automorphism groups containing many conjugacy classes of semiregular cyclic subgroups.

Computationally, the most problematic part of the algorithm proves to be Line 15, where all Hamiltonian cycles of the quotient graph $Q$ need to be determined.
For small enough graphs $Q$ (that is, for the cases where the potential rotation $g$ has few vertex-orbits), this can be done quickly with a brute-force search,
and can be made even more efficient for subcubic graphs (which was the case in our test case, presented in Section~\ref{sec:CET}). An implementation of a
Hamiltonian cycles determination algorithm is available in {\tt nauty} \cite{nauty}, for example.

No matter how good the implementation is, a deterministic algorithm for finding all Hamiltonian cycles in $Q$ will fail when the graph $Q$ is large.
The first improvement of the algorithm addressing this issue to some extent is to note that not
 all Hamiltonian cycles of $Q$ need to be determined before we start testing them within the loop starting in Line 17 and ending in Line 23.
 Instead, the Hamiltonian cycles in $Q$ can be generated and tested one by one, finishing the process as soon as we find one which passes the test in Line 18. This modification can speed up the computations considerably when $Q$ contains such Hamiltonian cycles, but of course is useless when $Q$ possesses many Hamiltonian cycles none of which passes the test.

Alternatively, instead of a process that systematically generates the Hamiltonian cycles in $Q$ one by one, one could also use a fast heuristic algorithm that tries to
find a Hamiltonian cycle in $Q$, test each Hamiltonian cycle as it is found, and then repeat this process until either it finds one that passes the test or until a limit on the number of
tries is reached. In the latter case, we can assume with a reasonable probability that none of the Hamiltonian cycles will pass the test and then we can thus move on
to the next potential rotation from the set $\cR$. In this case (unless the algorithm succeeds with a different potential rotation of the same order),
no claim can be made about the exact value of the Hamiltonian compression factor of $\Gamma$ and only a lower bound on $\kappa(\Gamma)$ can be determined in the case where the algorithm subsequently succeeds with a potential rotation of smaller order.
 
 Our implementation of the algorithm used a combination of these approaches. For quotient graphs $Q$ that have at most a  prescribed number $M$ 
 of vertices of degree greater than $2$ (noting that vertices of degree $2$ can be suppressed when finding Hamiltonian cycles), we performed a brute-force search
 that determined all Hamiltonian cycles of $Q$. For quotients $Q$ having more than $M$ vertices of degree at least $3$, we used the implementation of a heuristic algorithm 
 that is available in {\tt nauty} package \cite{nauty} under the name {\tt hamheuristic.c}, which we ran a prescribed number $P$ times before giving up and moving on
 to the next potential rotation in $\cR$. Of course, one can experiment with different values of the parameters $M$ and $P$ to achieve best performance.
 In our test case presented in Section~\ref{sec:CET}, we set $M$ to be $114$, and $P$ to be $50$.
 
 Finally, let us comment briefly on the test given in Line~18 of the algorithm. Here
for a fixed  sequence $\bv_0\bv_1\ldots \bv_{m-1}$, we need to decide whether there exist shifts
$
    z_i\in \zeta(\bar v_i,\bar v_{i+1}), i\in\ZZ_m,
$
such that
the sum $\sum_{i=0}^{m-1}z_i$ (which can be computed modulo $k$) is coprime to $k$.
This condition should not be tested by enumerating all possible choices of the
shifts $z_i$, though, as that would clearly result in an exponential explosion of the sums that need to be tested.
 Instead, one should successively  compute the sets $R_0 = \{0\}$, $R_j$, $j=1,2, \ldots, m$, 
of all partial sums $r_{j-1}+z_{j-1} \in \ZZ_k$ with $r_{j-1} \in R_{j-1}$ and $z_{j-1} \in \zeta(\bar v_{j-1}\bar v_{j})$,
and for each element of $R_j$, store one predecessor from $r_{j-1} \in R_{j-1}$ from
which it was reached by adding an element  $z_{j-1} \in \zeta(\bar v_{j-1},\bar v_{j})$,
so that a successful choice of shifts can later be reconstructed by backtracking.

\section{Hamiltonian compression factors of cubic edge-transitive graphs}
\label{sec:CET}

We tested our algorithm on the classical database of highly symmetric
graphs, namely the extended version~\cite{ConPotFoster} (available
for download at \cite{ConderWebCAT,ConderWebCSS,FosterGraphSym})
of the Foster
census~\cite{Foster}. This census contains all cubic edge-transitive graphs on
at most \(10{,}000\) vertices. It consists of \(4858\) graphs, of which \(3815\)
are arc-transitive and \(1043\) are semisymmetric, that is, edge-transitive but
not vertex-transitive. The census is available online at~\cite{FosterGraphSym}
in \texttt{sparse6} format, as well as in a Magma-compatible format.

The fact that all
edge-transitive cubic graphs of order at most
\(10{,}000\), except for
 the two well-known non-Hamiltonian exceptions (the Petersen graph,
denoted in the census~\cite{FosterGraphSym} by \(\mathrm{CAT}(10,1)\), and the
Coxeter graph, denoted by \(\mathrm{CAT}(28,1)\)), are Hamiltonian
was established already in 2018 by the first-named author of this paper,
but no attempt at finding Hamiltonian cycles admitting the highest rotational symmetry
was made at that point. Here, we determine
a provably
correct value of the Hamiltonian compression factor, together with a
corresponding shortest LCF code, for all but \(98\) graphs.

Our implementation used a combination of the approaches described in
Section~\ref{sec:pseudo}. For quotients \(Q\) constructed in Line~7 of
Algorithm~\ref{alg:hcf} and having at most \(114\) vertices of degree greater than \(2\), we
exhaustively generated all Hamiltonian cycles of \(Q\) and tested which of them
lift to rotationally symmetric Hamiltonian cycles of the original cubic graph.
For the remaining quotients, we used the heuristic algorithm, which we ran
\(50\) times before giving up.

The initial computation, run on a single core of an M1 Apple Mac mini, took
approximately \(15\) hours and determined the exact value of the Hamiltonian
compression factor for all but \(109\) graphs. For a further \(11\) graphs, we
extended the range of quotients for which all Hamiltonian cycles were
enumerated, thereby confirming that the lower bound previously obtained was in
fact the exact value of the Hamiltonian compression factor. These additional
computations, run in parallel on the \(8\) CPU cores of the M1 Mac mini, took
approximately \(10\) hours. For the remaining \(98\) graphs, we strongly
believe that the lower bounds obtained are the exact values.
The Hamiltonian compression factors and the corresponding LCF codes obtained
in our computations are available for download from the census
webpage~\cite{FosterGraphSym} and from the accompanying Zenodo
repository~\cite{ZenodoLCF}. 

The data is given in two text files, one for the
arc-transitive graphs and one for the semisymmetric. In each of these two files,
each graph is represented by one line, such as:\\
\indent
$!\> 24\> 1\!\colon\> [-5,5,-9,7,-7,9]^4$
\\
More precisely, each line starts with one of the symbols `$!$', `$?$' or `$\#$':
the symbol `$!$' indicates that an optimal Hamiltonian cycle (and hence the exact value of
the Hamiltonian compression factor) was determined, the symbol `$?$' indicates that only a lower
bound on the compression factor was established, and the symbol `$\#$' means that the graph
is not Hamiltonian (recalling that there are only two such graphs in the database).
Then two integers follow, say $n$ and $i$, meaning that the ID of the graph in
the census~\cite{FosterGraphSym} is CAT$(n,i)$ in the arc-transitive case, or CSS$(n,i)$
in the semisymmetric case. After the separator `$\colon$', the actual LCF of the graph follows, where
the length of the sequence in the square brackets is the length $p$ of the period of the LCF code,
and the integer in the exponent is the compression factor of the corresponding Hamiltonian cycle.

In what follows, we present some statistical features of the Hamiltonian
compression factor for cubic edge-transitive graphs. For the purpose of this
statistical analysis, we treat the lower bounds obtained for the $98$ unresolved
arc-transitive graphs as the exact values of their compression factor, and we
assign the compression factor $0$ to the two non-Hamiltonian graphs (the Petersen
and Coxeter graphs). We consider the two families separately: the $3815$
arc-transitive graphs and the $1043$ semisymmetric graphs.

Let us first discuss the average and the median value for the Hamiltonian compression factors
of the graphs in the two datasets.
The mean value of $\kappa$ is approximately
$795$ for the arc-transitive graphs and $97$ for the semisymmetric ones, while the
corresponding medians are $117$ and $36$. The substantial gap between the mean and
the median in both cases is due to the fact that a relatively small number
of exceptionally compressible graphs (those admitting very short LCF periods) pulls
the average well above the typical value.

The global distribution of the compression factor is shown in
Figure~\ref{fig:scatter}, where each graph is represented by a single point whose
horizontal coordinate is its order $n$ and whose vertical coordinate is its
compression factor $\kappa$. Since $\kappa = n/p$ for a graph with shortest LCF
period $p$, the points line up along the rays
$\kappa = n/p$ emanating from the origin, one ray for each attainable period $p$.
(A few of these are indicated by dashed lines.) The arc-transitive graphs populate
the rays $p=1,2,3,\ldots$ densely, with a conspicuous concentration along the line
$\kappa = n/2$, whereas the semisymmetric graphs avoid the steepest rays
altogether: no point lies above the line $\kappa = n/6$.

\begin{figure}[ht]
\centering
\includegraphics[width=\textwidth]{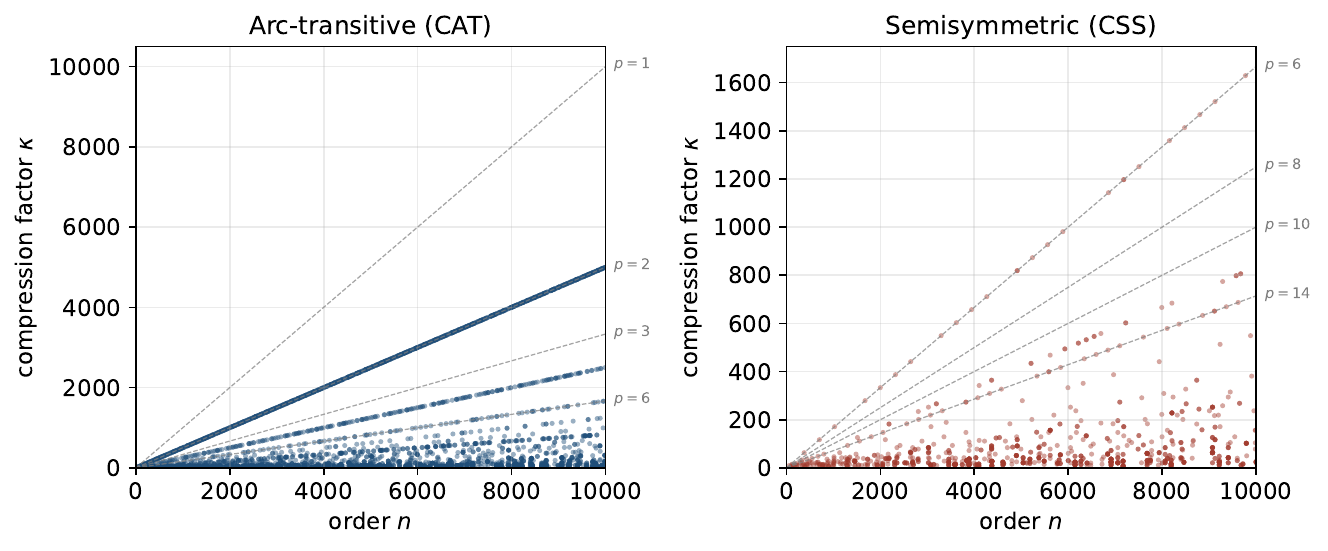}
\caption{The Hamiltonian compression factor $\kappa$ of each graph plotted against
its order $n$, for cubic arc-transitive graphs (left) and cubic semisymmetric
graphs (right). Each dot is one graph; the dashed rays are the lines
$\kappa = n/p$ for the indicated periods $p$. Note the different vertical scales,
and that no semisymmetric graph lies above the ray $\kappa = n/6$.}
\label{fig:scatter}
\end{figure}

Since $\kappa(\Gamma)$ is bounded above by the order $n$ of $\Gamma$, comparing the
Hamiltonian compression factor across graphs of different orders is more meaningful through the
{\em relative compression factor} $\kappa(\Gamma)/n$, which equals the reciprocal $1/p$ of the
period length $p$ of a shortest LCF code, and so lies in the interval $(0,1]$. A value
close to $1$ indicates an extremely compressible graph (having a short, highly repeated LCF
code), while a small value signals a long, nearly aperiodic code. Figure~\ref{fig:abs}
plots the average compression factor, and Figure~\ref{fig:rel} the average relative
compression factor, computed over all graphs of order at most $N$, as a function of
$N$. The near-linear growth of the average of $\kappa$ in Figure~\ref{fig:abs} is
to be expected and carries little information, since $\kappa$ is itself a quantity
that scales with the order $n$; the relative compression factor in
Figure~\ref{fig:rel} is the more meaningful one. There the average drops sharply
for small orders and then settles onto what appears to be an almost constant
plateau: approximately $0.16$ for arc-transitive graphs and approximately $0.02$
for semisymmetric graphs. Taken at face value, this would suggest that, across the
whole range up to order $10{,}000$, arc-transitive graphs are on average roughly
eight times more compressible than semisymmetric ones.

This last comparison should, however, be read with caution. The factor of eight is
the ratio of two {\em mean} relative compression factors, and, as already noted,
both distributions are strongly right-skewed, so the means are not representative
of a typical graph: the median relative compression factor is only about $0.04$ for
arc-transitive graphs, roughly a quarter of the corresponding mean, and is smaller
still for semisymmetric graphs. The ratio of the medians therefore differs appreciably
from the ratio of the means, and the qualitative conclusion---that arc-transitive
graphs are substantially more compressible---is more robust than the precise
numerical factor. We also stress that all of these statistics rest on the working
assumption stated above, namely that the lower bounds obtained for the $98$
unresolved arc-transitive graphs coincide with the true compression factors; should
any of these bounds be strict, the corresponding averages would only increase
slightly, and none of the qualitative conclusions would be affected.

\begin{figure}[ht!!!]
\centering
\includegraphics[width=\textwidth]{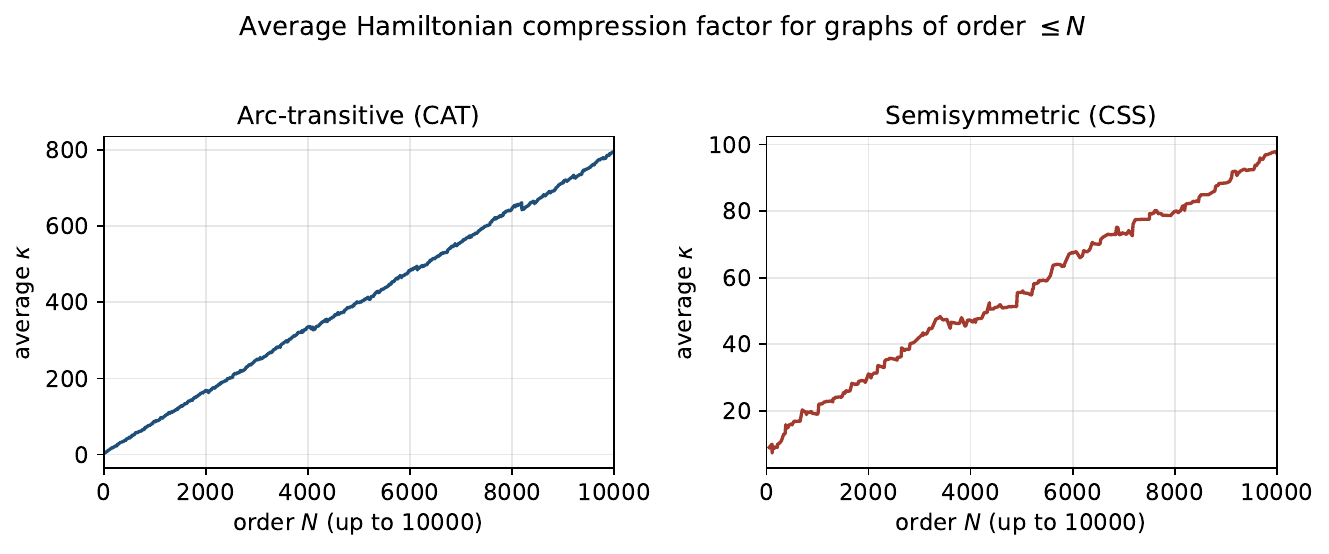}
\caption{The average Hamiltonian compression factor $\kappa$ over all edge-transitive cubic
graphs of order
at most $N$, plotted against $N$, for cubic arc-transitive graphs (left) and cubic
semisymmetric graphs (right).}
\label{fig:abs}
\end{figure}

\begin{figure}[ht!!!]
\centering
\includegraphics[width=\textwidth]{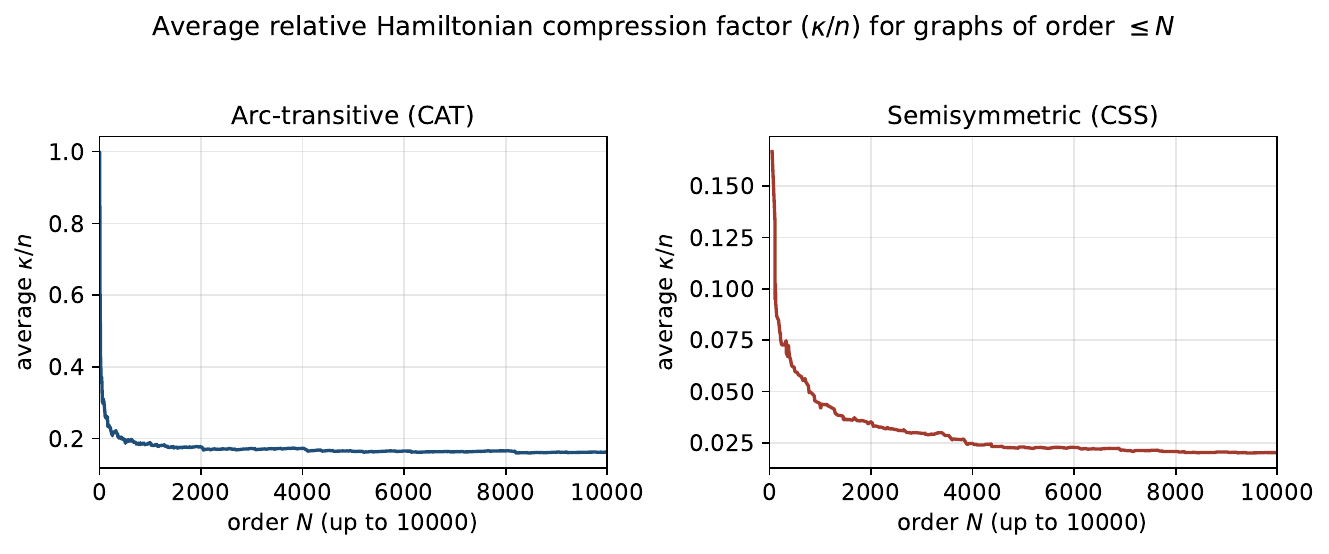}
\caption{The average relative Hamiltonian compression factor $\kappa/n$ over all edge-transitive cubic graphs
of order at most $N$, plotted against $N$, for cubic arc-transitive graphs (left) and
cubic semisymmetric graphs (right). In both families the average approaches an almost
constant value.}
\label{fig:rel}
\end{figure}

A further qualitative difference between the two families concerns the extremes of
compressibility. Among the arc-transitive graphs, $925$ of them (about $24\%$) admit an
LCF code of period at most $2$,  satisfying $\kappa(\Gamma)/n \ge 1/2$; two of
these, the complete graph $K_4 = \mathrm{CAT}(4,1)$ and the complete bipartite graph
$K_{3,3} = \mathrm{CAT}(6,1)$, even attain period $1$ and hence $\kappa(\Gamma)=n$. In
contrast, no semisymmetric graph in the census admits a period shorter than $6$: the
most compressible semisymmetric graphs satisfy $\kappa(\Gamma)/n = 1/6$, and
none exceeds that. This is precisely the empty region above the ray $\kappa = n/6$
visible in the right-hand panel of Figure~\ref{fig:scatter}. The lower
compressibility of semisymmetric graphs is hence not merely a
matter of averages, but holds even at the most favourable end of the spectrum.

\section{Concluding remarks}
\label{sec:conclusion}

We have presented an algorithm, justified by the cycle lifting lemma
(Lemma~\ref{lem:main}), which determines the Hamiltonian compression factor of a cubic
graph together with a correspondingly compressed LCF code, by exploiting the
semiregular automorphisms of the graph rather than enumerating Hamiltonian cycles
directly. Applied to the extended Foster census, the method determined a provably
optimal value of $\kappa(\Gamma)$ for all but $98$ of the $4858$ cubic edge-transitive
graphs on at most $10{,}000$ vertices, and a lower bound (believed to be exact) for the
remaining ones. The resulting LCF codes are made publicly available, both from the
census webpage~\cite{FosterGraphSym} and from the accompanying Zenodo
repository~\cite{ZenodoLCF}.

Several features of the data invite further study. First, every semisymmetric graph in
the census turned out to be Hamiltonian, and, perhaps more surprisingly, an optimal LCF
code was found for each of them; it would be interesting to know whether this remains
true beyond $10{,}000$ vertices. Second, the apparent stabilisation of the average
relative compression factor (Figure~\ref{fig:rel}) might suggest that this quantity 
converges to a limit as the order grows; either proving or disproving existence of such a limit
 for either family seems a natural question. Third, the marked distinction
between the two families (and in particular the empirical fact that no semisymmetric graph
in our range admits an LCF period shorter than $6$, whereas arc-transitive graphs
routinely admit periods $1$ or $2$) calls for a structural explanation. Finally, from
an algorithmic standpoint, the chief obstacle remains the determination of all
Hamiltonian cycles in the quotient graph for candidate rotations of small order; any
improvement here would directly extend the range in which the compression factor can be
computed exactly, and in particular might resolve the $98$ outstanding cases.
\bigskip

{\bf CRediT authorship contribution statement.}
\noindent
Marston Conder produced the census of arc-transitive cubic graphs of order up
to $10{,}000$ in 2011 and determined which of them admit a Hamiltonian cycle; in
2018, jointly with Primo\v{z} Poto\v{c}nik, he extended this census to
semisymmetric cubic graphs of the same order and showed that all of them are
Hamiltonian. He and Primo\v{z} Poto\v{c}nik also found compressed LCF codes for
many of the small graphs in these censuses (Resources, Investigation, Formal
analysis, Data curation). Primo\v{z} Poto\v{c}nik conceived and led the
theoretical part of the research, developing the underlying methodology and
carrying out the formal analysis, and he wrote the first draft of the
manuscript, in addition to supervising the project and overseeing its
administration (Conceptualization, Methodology, Formal analysis, Writing --
original draft, Supervision, Project administration). Gregor Poto\v{c}nik
implemented the algorithm in software, ran and managed the computations on the
cubic edge-transitive graphs, curated the resulting data, and validated the
computed compression factors and LCF codes (Software, Investigation, Data
curation, Validation).
\medskip

{\bf Acknowledgements.}
\noindent
We are grateful to Toma{\v z} Pisanski and Vlad Pelekhaty for asking us about determining LCF codes 
for graphs in the census of arc-transitive cubic graphs of small order, as well as for many discussions that
followed, and to Eric W.\ Weisstein for encouraging us to document our computation of these codes
in the format of a paper.

\end{document}